\documentclass[a4paper,10pt]{amsart}
\usepackage[utf8x]{inputenc}
\usepackage{amsthm}
\usepackage{amssymb}
\usepackage{hyperref}
\usepackage{graphicx}
\title[On floors and ceilings of the $k$-Catalan arrangement]{On floors and ceilings of the $k$-Catalan arrangement}
\author{Marko Thiel}
\newtheorem{theorem}{Theorem}
\numberwithin{theorem}{section}
\newtheorem{lemma}[theorem]{Lemma}
\newtheorem{proposition}[theorem]{Proposition}
\newtheorem{claim}{Claim}

\newtheorem{corollary}[theorem]{Corollary}

\newtheorem*{theorem*}{Theorem}
\newtheorem*{corollary*}{Corollary}
\def\mfl{U(M)}
\def\mcl{L(M)}
\def\I{\mathcal{I}}
\def\J{\mathcal{J}}
\def\sup{\mathsf{supp}}
\begin{document}

\begin{abstract} 
The set of dominant regions of the $k$-Catalan arrangement of a crystallographic root system $\Phi$ is a well-studied object enumerated by the Fu{\ss}-Catalan number $Cat^{(k)}(\Phi)$.
It is natural to refine this enumeration by considering floors and ceilings of dominant regions.
A conjecture of Armstrong states that counting dominant regions by their number of floors of a certain height gives the same distribution as counting dominant regions by their number of ceilings of the same height.
We prove this conjecture using a bijection that provides even more refined enumerative information.

\end{abstract}
\maketitle

\section{Introduction}
Let $\Phi$ be a crystallographic root system of rank $n$ with simple system $S$, positive system $\Phi^+$, and ambient vector space $V$. For background on root systems see \cite{humphreys90reflection}. For $k$ a positive integer, we define the \emph{$k$-Catalan arrangement} of $\Phi$ as the hyperplane arrangement given by the hyperplanes $H_{\alpha}^r=\{x\in V\mid\langle x,\alpha\rangle=r\}$ for $\alpha\in\Phi$ and $r\in\{0,1,\ldots,k\}$.
The complement of this arrangement falls apart into connected components which we call the \emph{regions} of the arrangement.
Those regions $R$ that have $\langle x,\alpha\rangle>0$ for all $\alpha\in\Phi^+$ and all $x\in R$ we call \emph{dominant}. The number of dominant regions of the $k$-Catalan arrangement equals the Fu{\ss}-Catalan number $Cat^{(k)}(\Phi)$ \cite{athanasiadis04generalized} of $\Phi$. This number remains somewhat mysterious, in the sense that it also counts other objects in combinatorics, like the set of $k$-divisible noncrossing partitions $NC^{(k)}(\Phi)$ of $\Phi$ \cite[Theorem 3.5.3]{armstrong09generalized} and the number of facets of the $k$-generalised cluster complex $\Delta^{(k)}(\Phi)$ of $\Phi$ \cite[Proposition 8.4]{fomin05generalized}, but no uniform proof of this fact is known, that is every known proof of this fact appeals to the classification of irreducible crystallographic root systems.\\
\\
For a dominant region $R$ of the $k$-Catalan arrangement, we call those hyperplanes that support a facet of $R$ the \emph{walls} of $R$. Those walls of $R$ which do not contain the origin and have the origin on the same side as $R$ we call the \emph{ceilings} of $R$. The walls of $R$ that do not contain the origin and separate $R$ from the origin are called its \emph{floors}.
We say a hyperplane is of \emph{height} $r$ if it is of the form $H_{\alpha}^r$ for $\alpha\in\Phi^+$.\\
\\
One reason why floors and ceilings of dominant regions are interesting is that they give a more refined enumeration of the dominant regions of the $k$-Catalan arrangement of $\Phi$ that corresponds to refined enumerations of other objects counted by the Fu{\ss}-Catalan number $Cat^{(k)}(\Phi)$.
More precisely, the number of dominant regions in the $k$-Catalan arrangement of $\Phi$ that have exactly $j$ floors of height $k$ equals the Fu{\ss}-Narayana number $Nar^{(k)}(\Phi,j)$ \cite[Proposition 5.1]{athanasiadis05refinement} \cite[Theorem 1]{thiel13hf}, which also counts the number of $k$-divisible noncrossing partitions of $\Phi$ of rank $j$ \cite[Definition 3.5.4]{armstrong09generalized}, as well as equalling the $(n-j)$-th entry of the $h$-vector of the $k$-generalised cluster complex $\Delta^{(k)}(\Phi)$ \cite[Theorem 10.2]{fomin05generalized}.
Similarly, the number of bounded dominant regions of the $k$-Catalan arrangement of $\Phi$ that have exactly $j$ ceilings of height $k$ equals the $(n-j)$-th entry of the $h$-vector of the positive part of $\Delta^{(k)}(\Phi)$ \cite[Conjecture 1.2]{athanasiadis06cluster} \cite[Corollary 5]{thiel13hf}.\\
\\
For the special case where $\Phi$ is of type $A_{n-1}$, more is known. For example, there is an explicit bijection between the set of dominant regions of the $k$-Catalan arrangement of $\Phi$ and the set of facets of the cluster complex of $\Phi$ \cite{fishel13facets}.
There is also an enumeration of those dominant regions that have a fixed hyperplane as a floor \cite{fishel13fixed}. In contrast to those results, all results in this paper are stated and proven uniformly for all crystallographic root systems without appeal to the classification.\\
\\
If $M$ is any set of hyperplanes of the $k$-Catalan arrangement, let $\mfl$ be the set of dominant regions $R$ of the $k$-Catalan arrangement such that all hyperplanes in $M$ are floors of $R$.
Similarly, let $\mcl$ be the set of dominant regions $R'$ of the $k$-Catalan arrangement such that all hyperplanes in $M$ are ceilings of $R'$. 
Use the standard notation $[n]:=\{1,2,\ldots,n\}$.
Then we have the following theorem.
\begin{theorem}\label{bij}
 For any set $M=\{H_{\alpha_1}^{i_1},H_{\alpha_2}^{i_2},\ldots,H_{\alpha_m}^{i_m}\}$ of $m$ hyperplanes with $i_j\in [k]$ and $\alpha_j\in\Phi^+$ for all $j\in[m]$, there is an explicit bijection $\Theta$ from $\mfl$ to $\mcl$.
\end{theorem}
See Figure \ref{theta} for an example. From this theorem, we obtain some enumerative corollaries. In particular, let $fl_r(l)$ be the number of dominant regions in the $k$-Catalan arrangement that have exactly $l$ floors of height $r$, and let $cl_r(l)$ be the number of dominant regions that have exactly $l$ ceilings of height $r$ \cite[Definition 5.1.23]{armstrong09generalized}.
We deduce the following conjecture of Armstrong.
\begin{figure}\label{theta}
\begin{center}
 \resizebox*{12.4cm}{!}{\includegraphics{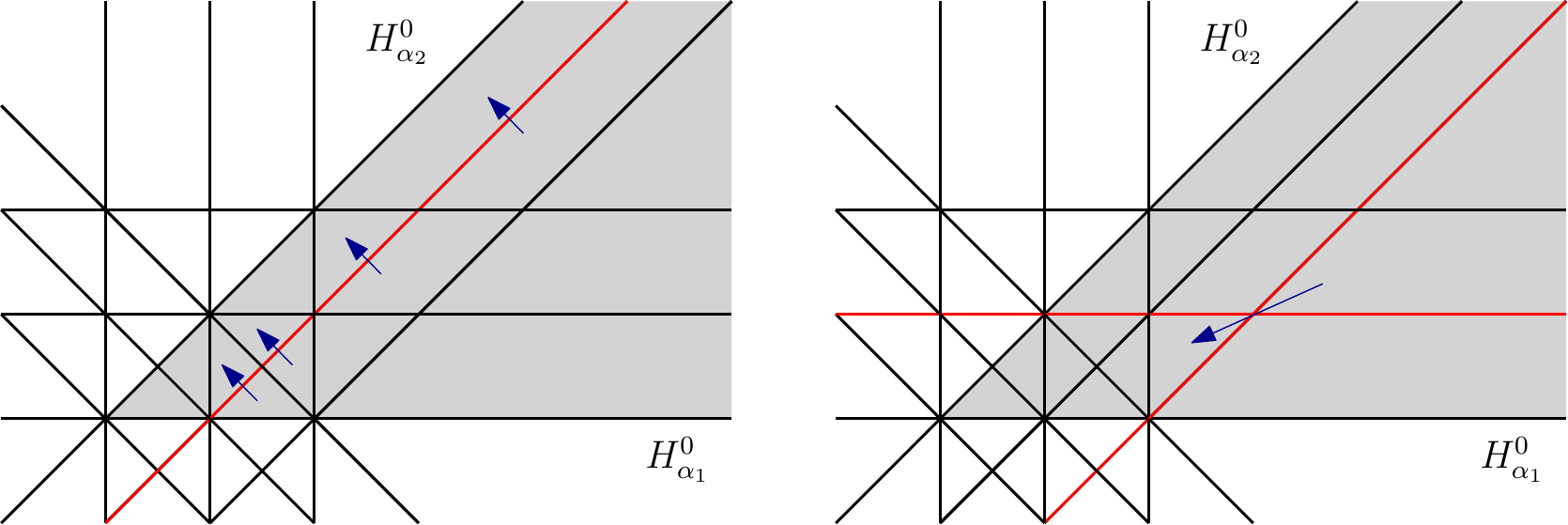}}\\
\end{center}
\caption{The bijection $\Theta$ for the 2-Catalan arrangement of the root system of type $B_2$, for $M=\{H_{\alpha_2}^1\}$ and for $M=\{H_{\alpha_1}^1,H_{\alpha_2}^2\}$. The dominant chamber is shaded in grey.}
\end{figure}

\begin{corollary}[\protect{\cite[Conjecture 5.1.24]{armstrong09generalized}}]\label{arm}
  We have $fl_r(l)=cl_r(l)$ for all $1\leq r\leq k$ and $0\leq l\leq n$.
\end{corollary}
Specialising to the $k=1$ case, we also give a geometric interpretation in terms of dominant regions of the Catalan arrangement of the Panyushev complement on ideals in the root poset of $\Phi$.
\section{Definitions}
For this section and the next one, suppose that $\Phi$ is irreducible. Define the \emph{affine Coxeter arrangement} of $\Phi$ as the union of all hyperplanes of the form $H_{\alpha}^r=\{x\in V\mid\langle x,\alpha\rangle=r\}$ for $\alpha\in\Phi$ and $r\in\mathbb{Z}$.
Then the complement of this falls apart into connected components, all of which are congruent open $n$-simplices, called \emph{alcoves}.
The \emph{affine Weyl group} $W_a$ generated by all the reflections through hyperplanes of the form $H_{\alpha}^r$ for $\alpha\in\Phi$ and $r\in\mathbb{Z}$ is a Coxeter group, with generating set $S_a=\{s_0,s_1,\ldots,s_n\}$, where $s_1,\ldots,s_n$ are the reflections in the hyperplanes orthogonal to the simple roots of $\Phi$ and $s_0$ is the reflection in $H^1_{\tilde{\alpha}}$, where $\tilde{\alpha}$ is the highest root of $\Phi$.\\
\\
The group $W_a$ acts simply transitively on the alcoves, so if we define the \emph{fundamental alcove} as
$$A_{\circ}=\{x\in V\mid \langle x,\alpha_i\rangle >0\text{ for all }\alpha_i\in S, \langle x,\tilde{\alpha}\rangle<1\}\text{,}$$
then every alcove $A$ can be written as $w(A_{\circ})$ for a unique $w\in W_a$.\\
\\
Clearly any alcove is contained in exactly one region $R$ of the $k$-Catalan arrangement of $\Phi$. For any alcove $A$ in the affine Coxeter arrangement of $\Phi$ and $\alpha\in\Phi^+$, there exists a unique integer $r$ with $r-1<\langle x,\alpha\rangle<r$ for all $x\in A$. We denote this integer by $r(A,\alpha)$.\\
\\
Suppose that for each $\alpha\in\Phi^+$ we are given a positive integer $r_{\alpha}$. The following is due to Shi \cite[Theorem 5.2]{shi87alcoves}.
\begin{lemma}[\protect{\cite[Lemma 2.3]{athanasiadis06cluster}}]\label{aff}
 There is an alcove $A$ with $r(A,\alpha)=r_{\alpha}$ for all $\alpha\in\Phi^+$ if and only if $r_{\alpha}+r_{\beta}-1\leq r_{\alpha+\beta}\leq r_{\alpha}+r_{\beta}$ whenever $\alpha,\beta,\alpha+\beta\in\Phi^+$.
\end{lemma}

Define a partial order on $\Phi^+$ by 
$$\alpha\leq\beta\text{ if and only if } \beta-\alpha\in \langle S\rangle_{\mathbb{N}}\text{,}$$
that is, $\beta\geq\alpha$ if and only if $\beta-\alpha$ can be written as a linear combination of simple roots with nonnegative integer coefficients.
The set of positive roots $\Phi^+$ with this partial order is called the \emph{root poset}. A subset $I\subseteq\Phi^+$ is called an \emph{ideal} if for all $\alpha\in I$ and $\beta\leq\alpha$, also $\beta\in I$.
A subset $J\subseteq\Phi^+$ is called an \emph{order filter} if for all $\alpha\in J$ and $\beta\geq\alpha$, also $\beta\in J$.\\
\\
Suppose $\I=(I_1,I_2,\ldots,I_k)$ is an ascending (multi)chain of $k$ ideals in the root poset of $\Phi$, that is $I_1\subseteq I_2\subseteq\ldots\subseteq I_k$.
Setting $J_i=\Phi^+\backslash I_i$ for $i\in[k]$ and $\J=(J_1,J_2,\ldots,J_k)$ gives us the corresponding descending chain of order filters. That is, we have $J_1\supseteq J_2\supseteq\ldots\supseteq J_k$.
The ascending chain of ideals $\I$ and the corresponding descending chain of order filters $\J$ are both called \emph{geometric} if the following conditions are satisfied simultaneously.
\begin{enumerate}
 \item $(I_i+I_j)\cap\Phi^+\subseteq I_{i+j}\text{ for all }i,j\in\{0,1,\ldots,k\}\text{ with }i+j\leq k\text{, and}$
 \item $(J_i+J_j)\cap\Phi^+\subseteq J_{i+j}\text{ for all }i,j\in\{0,1,\ldots,k\}\text{.}$
\end{enumerate}
Here we set $I_0=\varnothing$, $J_0=\Phi^+$, and $J_i=J_k$ for $i>k$.
We call $\I$ and $\J$ \emph{positive} if $S\subseteq I_k$, or equivalently $S\cap J_k=\varnothing$.\\
\\
Let $R$ be a dominant region of the $k$-Catalan arrangement of $\Phi$. 
Define $\theta(R)=(I_1,I_2,\ldots,I_k)$ and $\phi(R)=(J_1,J_2,\ldots,J_k)$, where
$$I_i=\{\alpha\in\Phi^+\mid\langle x,\alpha\rangle<i\text{ for all }x\in R\}\text{ and}$$
$$J_i=\{\alpha\in\Phi^+\mid\langle x,\alpha\rangle>i\text{ for all }x\in R\}\text{,}$$
for $i\in\{0,1,\ldots,k\}$.
It is not difficult to verify that $\theta(R)$ is a geometric chain of ideals and that $\phi(R)$ is the corresponding geometric chain of order filters.\\
%
\\
For a geometric chain of ideals $\I=(I_1,I_2,\ldots,I_k)$, and $\alpha\in\Phi^+$, we define
$$r_{\alpha}(\I)=\text{min}\{r_1+r_2+\ldots+r_m\mid\alpha=\alpha_1+\alpha_2+\ldots+\alpha_m\text{ and }\alpha_i\in I_{r_i}\text{ for all }i\in[m]\}\text{,}$$
where we set $r_{\alpha}(\I)=\infty$ if $\alpha$ cannot be written as a linear combination of elements in $I_k$.
So $r_{\alpha}(\I)<\infty$ for all $\alpha\in\Phi^+$ if and only if $\I$ is positive.\\
\\
For a geometric chain of order filters $\J=(J_1,J_2,\ldots,J_k)$, and $\alpha\in\Phi^+$, we define
$$k_{\alpha}(\J)=\text{max}\{k_1+k_2+\ldots+k_m\mid\alpha=\alpha_1+\alpha_2+\ldots+\alpha_m\text{ and }\alpha_i\in J_{k_i}\text{ for all }i\in[m]\}\text{,}$$
where $k_i\in\{0,1,\ldots,k\}$ for all $i\in[m]$.\\
\\
It turns out that $\phi$ is a bijection from the set of dominant regions of the $k$-Catalan arrangement of $\Phi$ to the set of geometric chains of $k$ order filters in the root poset of $\Phi$ \cite[Theorem 3.6]{athanasiadis05refinement}.
Its inverse $\psi$ is the map sending a geometric chain of order filters $\J$ to the region $R$ of the $k$-Catalan arrangement containing the alcove $A$ with $r(A,\alpha)=k_{\alpha}(\J)+1$ for all $\alpha\in\Phi^+$. This alcove $A$ is called the \emph{minimal alcove} of $R$. Its floors are exactly the floors of $R$ \cite[Theorem 3.11]{athanasiadis05refinement}.\\
\\
Thus the map $\theta$ is a bijection from dominant regions $R$ of the $k$-Catalan arrangement to geometric chains of ideals $\I$. It restricts to a bijection between bounded dominant regions of the $k$-Catalan arrangement and positive geometric chains of ideals. 
The inverse of this restriction maps a positive geometric chain of ideals $\I$ to the bounded dominant region $R$ in the $k$-Catalan arrangement containing the alcove $B$ with $r(B,\alpha)=r_{\alpha}(\I)$ for all $\alpha\in\Phi^+$ \cite[Theorem 3.6]{athanasiadis06cluster}. This alcove $B$ is called the \emph{maximal alcove} of $R$. Its ceilings are exactly the ceilings of $R$ \cite[Theorem 3.11]{athanasiadis06cluster}.\\
\\
We call $\alpha\in\Phi^+$ a \emph{rank $r$ indecomposable element} \cite[Definition 3.8]{athanasiadis05refinement} of a geometric chain of order filters $\J=(J_1,J_2,\ldots,J_k)$ if $\alpha\in J_r$ and
\begin{enumerate}
 \item $k_{\alpha}(\J)=r$,
 \item $\alpha\notin J_i+J_j\text{ for }i+j=r$ and
 \item if $k_{\alpha+\beta}(\J)=t\leq k$ for some $\beta\in\Phi^+$ then $\beta\in J_{t-r}$.
\end{enumerate}
\vphantom{Fnord}
We have that $H^r_{\alpha}$ is a floor of $R$ if and only if $\alpha$ is a rank $r$ indecomposable element of the geometric chain of order filters $\J=\phi(R)$ \cite[Theorem 3.11]{athanasiadis05refinement}.\\
\\
We call $\alpha\in\Phi^+$ a \emph{rank $r$ indecomposable element} \cite[Definition 3.8]{athanasiadis06cluster} of a geometric chain of ideals $\I=(I_1,I_2,\ldots,I_k)$ if $\alpha\in I_r$ and
\begin{enumerate}
 \item $r_{\alpha}(\I)=r$,
 \item $\alpha\notin I_i+I_j\text{ for }i+j=r$ and
 \item if $r_{\alpha+\beta}(\I)=t\leq k$ for some $\beta\in\Phi^+$ then $\beta\in I_{t-r}$.
\end{enumerate}
\vphantom{Fnord}
We will soon see that $H^r_{\alpha}$ is a ceiling of $R$ if and only if $\alpha$ is a rank $r$ indecomposable element of the geometric chain of ideals $\I=\theta(R)$.
\section{Lemmas}
Our aim for this rather technical section is to prove the following theorem.
\begin{theorem}\label{ind=ceil}
 Let $R$ be a dominant region in the $k$-Catalan arrangement of $\Phi$, $\I=\theta(R)$ and $\alpha\in\Phi^+$.
 Then $R$ contains an alcove $B$ such that for all $r\in[k]$ the following are equivalent:
 \begin{enumerate}
  \item $H_{\alpha}^r$ is a ceiling of $R$,
  \item $\alpha$ is a rank $r$ indecomposable element of $\I$, and
  \item $H_{\alpha}^r$ is a ceiling of $B$.
 \end{enumerate}

\end{theorem}
It is already known that Theorem \ref{ind=ceil} holds for bounded dominant regions \cite[Theorem 3.11]{athanasiadis06cluster}.
In that case, we may take the alcove $B$ to be the maximal alcove of the bounded region $R$.\\
\\
Our approach to proving Theorem \ref{ind=ceil} is to note that when a region $R$ of the $k$-Catalan arrangement is subdivided into regions of the $(k+1)$-Catalan arrangement by hyperplanes of the form $H_{\alpha}^{k+1}$ for $\alpha\in\Phi^+$,
at least one of the resulting regions is bounded. We find a region $\underline{R}$ of the $(k+1)$-Catalan arrangement which, among the bounded regions of the $(k+1)$-Catalan arrangement that are contained in $R$, is the one furthest away from the origin.
We call the maximal alcove $B$ of $\underline{R}$ the \emph{pseudomaximal} alcove of $R$. It equals the maximal alcove of $R$ if $R$ is bounded. The alcove $B\subseteq R$ will be seen to satisfy the assertion of Theorem \ref{ind=ceil}. Instead of working directly with the dominant regions of the $k$- and $(k+1)$-Catalan arrangements, we usually phrase our results in terms of the corresponding geometric chains of ideals.
\begin{figure}[h]
\begin{center}
 \includegraphics{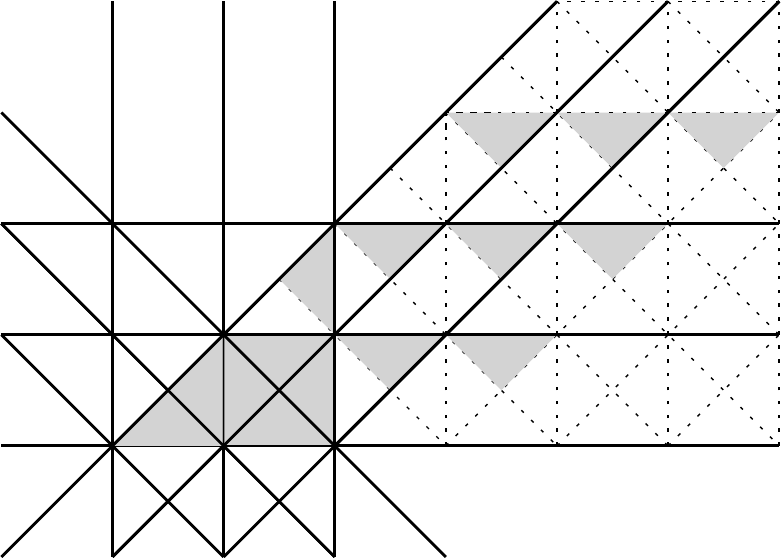}\\
\end{center}
\caption{The dominant regions of the 2-Catalan arrangement of the root system of type $B_2$ together with their pseudomaximal alcove, shaded in grey.}
\end{figure}

We require the following lemmas:
\begin{lemma}[\protect{\cite[Lemma 2.1 (ii)]{athanasiadis05refinement}}]\label{2.1}
 If $\alpha_1,\alpha_2,\ldots,\alpha_r\in\Phi$ and $\alpha_1+\alpha_2+\ldots+\alpha_r=\alpha\in\Phi$, then $\alpha_1=\alpha$ or there exists $i$ with $2\leq i\leq r$ such that $\alpha_1+\alpha_i\in\Phi\cup\{0\}$.
\end{lemma}

\begin{lemma}[\protect{\cite[Lemma 3.2]{athanasiadis06cluster}}]\label{inI}
 For $\alpha\in\Phi^+$ and $r_{\alpha}(\I)=r\leq k$, we have that $\alpha\in I_{r}$.
\end{lemma}

\begin{lemma}[\protect{\cite[Lemma 3.10]{athanasiadis06cluster}}]\label{rind}
 Suppose $\alpha$ is an indecomposable element of $\I$. Then
 \begin{enumerate}
  \item $r_{\alpha}(\I)=r_{\beta}(\I)+r_{\gamma}(\I)-1$ if $\alpha=\beta+\gamma$ for $\beta,\gamma\in\Phi^+$ and
  \item $r_{\alpha}(\I)+r_{\beta}(\I)=r_{\alpha+\beta}(\I)$ if $\beta,\alpha+\beta\in\Phi^+$.
 \end{enumerate}

\end{lemma}




\begin{lemma}\label{ideal}
 If $\alpha,\beta,\gamma\in\Phi^+$, $\beta+\gamma\in\Phi^+$ and $\alpha\leq\beta+\gamma$, then $\alpha\leq\beta$ or $\alpha\leq\gamma$ or $\alpha=\beta'+\gamma'$ with $\beta',\gamma'\in\Phi^+$, $\beta'\leq\beta$ and $\gamma'\leq\gamma$.
 \begin{proof}
  Let $\alpha=\beta+\gamma-\sum_{j\in J}\alpha_j$ with $\alpha_j\in S$ for all $j\in J$. We proceed by induction on $|J|$.
  If $|J|=0$, we are done. If $|J|=1$, we have that $\alpha=-\alpha_i+\beta+\gamma$ for some $\alpha_i\in S$.
  Thus by Lemma \ref{2.1}, we have either $\alpha=-\alpha_i$ (a contradiction), or $\beta'=\beta-\alpha_i\in\Phi\cup\{0\}$ or $\gamma'=\gamma-\alpha_i\in\Phi\cup\{0\}$.
  Notice that if $\beta'\neq0$, then $\beta'\in\Phi^+$, and similarly for $\gamma'$. So if $\beta'\in\Phi^+$ we may write $\alpha=\beta'+\gamma$ and otherwise we have $\gamma'\in\Phi^+$ and thus $\alpha=\beta+\gamma'$ as required.\\
  \\
  If $|J|>1$, we have $\alpha+\sum_{j\in J}\alpha_j=\beta+\gamma$, so by Lemma \ref{2.1}, either $\alpha=\beta+\gamma$, so we are done, or $\alpha+\alpha_j\in\Phi\cup\{0\}$ for some $j\in J$. In the latter case we even have $\alpha+\alpha_j\in\Phi^+$. By induction hypothesis, $\alpha+\alpha_j\leq\beta$ or $\alpha+\alpha_j\leq\gamma$ or $\alpha+\alpha_j=\beta'+\gamma'$ with $\beta',\gamma'\in\Phi^+$, $\beta'\leq\beta$ and $\gamma'\leq\gamma$. In the first two cases, we are done.
  In the latter case, we have $\alpha=-\alpha_j+\beta'+\gamma'$, so we proceed as in the $|J|=1$ case.
 \end{proof}

\end{lemma}

We are now ready to define the bounded dominant region $\underline{R}$ of the $(k+1)$-Catalan arrangement in terms of the corresponding geometric chain of $k+1$ ideals $\underline{\I}$. For a geometric chain of ideals $\I=(I_1,I_2,\ldots,I_k)$, let $\underline{I}_i=I_i$ for all $i\in [k]$ and let $\underline{I}_{k+1}=\bigcup_{i+j=k+1}((I_i+I_j)\cap\Phi^+)\cup I_k\cup S$. By Lemma \ref{ideal}, $\underline{I}_{k+1}$ is an ideal. Define $\underline{\I}=(\underline{I}_1,\ldots,\underline{I}_{k+1})$.
\begin{lemma}\label{Ibar}
 If $\I=(I_1,I_2,\ldots,I_k)$ is a geometric chain of $k$ ideals in the root poset of $\Phi$, then $\underline{\I}$ is a positive geometric chain of $k+1$ ideals. The bounded dominant region $\underline{R}=\theta^{-1}(\underline{\I})$ of the $(k+1)$-Catalan arrangement of $\Phi$ is contained in the region $R=\theta^{-1}(\I)$ of the $k$-Catalan arrangement.
 \begin{proof}
  By construction, $\underline{\I}$ is an ascending chain of ideals. If $i+j\leq k$, we have that $(\underline{I}_i+\underline{I}_j)\cap\Phi^+=(I_i+I_j)\cap\Phi^+\subseteq I_{i+j}=\underline{I}_{i+j}$ as $\I$ is geometric.
  If $i+j=k+1$ with $i,j\neq0$ (otherwise the result is trivial) we have that $(\underline{I}_i+\underline{I}_j)\cap\Phi^+=(I_i+I_j)\cap\Phi^+\subseteq\bigcup_{i+j=k+1}((I_i+I_j)\cap\Phi^+)\subseteq\underline{I}_{i+j}$.\\
  \\
  Let $\J=(J_1,J_2,\ldots,J_k)$ be the geometric chain of order filters corresponding to the geometric chain of ideals $\I$.
  Define $\underline{\J}$ similarly. 
  We need to verify that $(\underline{J}_i+\underline{J}_j)\cap\Phi^+\subseteq \underline{J}_{i+j}$ for all $i,j\in[k+1]$.\\
  \\
  Suppose first that $i+j\leq k$. Then $(\underline{J}_i+\underline{J}_j)\cap\Phi^+=(J_i+J_j)\cap\Phi^+\subseteq J_{i+j}=\underline{J}_{i+j}$ since $\J$ is geometric.\\
  \\
  Suppose next that $i+j=k+1$. Take any region $R'$ of the $(k+1)$-Catalan arrangement that is contained in $R$. Let $\theta(R')=\I'=(I_1',I_2',\ldots,I_{k+1}')$ be the geometric chain of ideals corresponding to $R'$ and let $\J'=(J_1',J_2',\ldots,J_{k+1}')$ be the corresponding geometric chain of order filters.
  Then $R$ and $R'$ are on the same side of each hyperplane of the $k$-Catalan arrangement. Thus $I_l'=I_l$ and $J_l'=J_l$ for $l\in[k]$. 
  Thus we have $\underline{I}_{k+1}=\bigcup_{i+j=k+1}((I_i+I_j)\cap\Phi^+)\cup I_k\cup S=\bigcup_{i+j=k+1}((I'_i+I'_j)\cap\Phi^+)\cup I'_k\cup S\subseteq I_{k+1}'\cup S$ since $\I'$ is geometric.
  Since $\J'$ is geometric, we have $(\underline{J}_i+\underline{J}_j)\cap\Phi^+=(J_i'+J_j')\cap\Phi^+\subseteq J_{i+j}'=J_{k+1}'$.
  The sum of two positive roots is never a simple root, so we even have $(\underline{J}_i+\underline{J}_j)\cap\Phi^+\subseteq J_{k+1}'\backslash{S}$. But $J_{k+1}'\backslash{S}\subseteq \underline{J}_{k+1}$, as $\underline{I}_{k+1}\subseteq I'_{k+1}\cup S$.
  Thus $(\underline{J}_i+\underline{J}_j)\cap\Phi^+\subseteq\underline{J}_{i+j}$.\\
  \\
  Lastly, in the case where $i+j>k+1$, we have $\underline{J}_j\subseteq\underline{J}_{k+1-i}$, so that $(\underline{J}_i+\underline{J}_j)\cap\Phi^+\subseteq(\underline{J}_i+\underline{J}_{k+1-i})\cap\Phi^+\subseteq\underline{J}_{k+1}=\underline{J}_{i+j}$.\\
  \\
  Thus the chain of ideals $\underline{\I}$ is geometric. It is also clearly positive, so $\underline{R}=\theta^{-1}(\underline{\I})$ is bounded. Since $\underline{I}_i=I_i$ for $i\in[k]$, $\underline{R}$ and $R$ are on the same side of each hyperplane of the $k$-Catalan arrangement, so $\underline{R}$ is contained in $R$.
 \end{proof}

\end{lemma}

For a geometric chain of $k$ ideals $\I=(I_1,I_2,\ldots,I_k)$, define $\sup(\I)=I_k\cap S$.
In particular, $\sup(\I)=S$ if and only if $\I$ is positive.
\begin{lemma}\label{r=r3}
 If $\alpha\in\langle \sup(\I)\rangle_{\mathbb{N}}$, then $r_{\alpha}(\underline{\I})=r_{\alpha}(\I)$. In particular, if $r_{\alpha}(\underline{\I})\leq k$, then $r_{\alpha}(\underline{\I})=r_{\alpha}(\I)$.
 \begin{proof}
  First note that $\alpha\in\langle \sup(\I)\rangle_{\mathbb{N}}$ implies that $r_{\alpha}(\I)<\infty$.
  So may write $\alpha=\alpha_1+\alpha_2+\ldots+\alpha_m$ with $\alpha_i\in I_{r_i}$ for $i\in[m]$ and $r_1+r_2+\ldots+r_m=r_{\alpha}(\I)$.
  Since $\alpha_i\in I_{r_i}=\underline{I}_{r_i}$ this implies that $r_{\alpha}(\underline{\I})\leq r_{\alpha}(\I)$.\\
  \\
  We may write $\alpha=\alpha_1+\alpha_2+\ldots+\alpha_m$ with $\alpha_i\in \underline{I}_{r_i}$ for $i\in[m]$ and $r_1+r_2+\ldots+r_m=r_{\alpha}(\underline{\I})$.
  We wish to show that $r_{\alpha}(\I)\leq r_{\alpha}(\underline{\I})$.
  Thus we seek to write $\alpha=\alpha_1'+\alpha_2'+\ldots+\alpha_l'$ with $\alpha_i'\in I_{r_i'}$ for $i\in[l]$ and $r_1'+r_2'+\ldots+r_l'=r_{\alpha}(\underline{\I})$.
  If $r_p=k+1$ for some $p\in[m]$, then $\alpha_p\in\underline{I}_{k+1}=\bigcup_{i+j=k+1}((I_i+I_j)\cap\Phi^+)\cup I_k\cup S$.
  If $\alpha_p\in I_k=\underline{I}_k$, we get a contradiction with the minimality of $r_{\alpha}(\underline{I})$.
  If $\alpha_p\in S$, then since $\alpha_p\in\langle \sup(\I)\rangle_{\mathbb{N}}$, we have that $\alpha_p\in\sup(\I)\subseteq I_k$, again a contradiction.
  So $\alpha_p\in\bigcup_{i+j=k+1}((I_i+I_j)\cap\Phi^+)$.
  Thus write $\alpha_p=\beta_p+\beta_p'$, where $\beta_p\in I_i$ and $\beta_p'\in I_j$ for some $i,j$ with $i+j=k+1$.
  So in the sum $\alpha=\alpha_1+\alpha_2+\ldots+\alpha_m$ replace each $\alpha_p$ with $r_p=k+1$ with $\beta_p+\beta_p'$ to obtain (after renaming) $\alpha=\alpha_1'+\alpha_2'+\ldots+\alpha_l'$ with $\alpha_i'\in I_{r_i'}$ for $i\in[l]$ and $r_1'+r_2'+\ldots+r_l'=r_{\alpha}(\underline{\I})$, as required.\\
  \\
  If $r_{\alpha}(\underline{\I})=r\leq k$, then $\alpha\in I_r\subseteq I_k$ by Lemma \ref{inI}, so $\alpha\in\langle \sup(\I)\rangle_{\mathbb{N}}$ and thus $r_{\alpha}(\underline{\I})=r_{\alpha}(\I)$.  
 \end{proof}
\end{lemma}
For $R$ a dominant region of the $k$-Catalan arrangement, define the \emph{pseudomaximal} alcove of $R$ to be the maximal alcove of $\underline{R}$.
This term is justified by the following proposition.
\begin{proposition}
 If $R$ is a bounded dominant region of the $k$-Catalan arrangement, its pseudomaximal alcove is equal to its maximal alcove.
 \begin{proof}
  Let $A$ and $B$ be the maximal and pseudomaximal alcoves of $R$ respectively. If $\I=\theta(R)$, then $r(\alpha,A)=r_{\alpha}(\I)$ for all $\alpha\in\Phi^+$.
  Since $B$ is the maximal alcove of $\underline{R}$, we have $r(\alpha,B)=r_{\alpha}(\underline{\I})$ for all $\alpha\in\Phi^+$.
  Now $\I$ is positive since $R$ is bounded, so $\sup(\I)=S$. Thus $r_{\alpha}(\I)=r_{\alpha}(\underline{\I})$ for all $\alpha\in\Phi^+$ by Lemma \ref{r=r3}.
  So $r(\alpha,A)=r(\alpha,B)$ for all $\alpha\in\Phi^+$ and therefore $A=B$.
 \end{proof}

\end{proposition}

\begin{lemma}\label{max}
 Let $R$ be a region of the $k$-Catalan arrangement of $\Phi$, let be $B$ be its pseudomaximal alcove and let $t\leq k$ be a positive integer.
 If $\langle x_0,\alpha\rangle>t$ for some $x_0\in R$, then $\langle x,\alpha\rangle>t$ for all $x\in B$.
 \begin{proof}
  Let $\I=\theta(R)$. Since $r(B,\alpha)=r_{\alpha}(\underline{\I})$ for all $\alpha\in\Phi^+$, it suffices to show that $r_{\alpha}(\underline{\I})>t$.
  If $r_{\alpha}(\underline{\I})>k$ this is immediate, so we may assume that $r_{\alpha}(\underline{\I})\leq k$. Thus we have $r_{\alpha}(\underline{\I})=r_{\alpha}(\I)$ by Lemma \ref{r=r3}.
  Write $\alpha=\alpha_1+\alpha_2+\ldots+\alpha_m$, with $\alpha_i\in I_{r_i}$ for all $i\in[m]$ and $r_1+r_2+\ldots+r_m=r_{\alpha}(\I)$.
  Then $\langle x,\alpha_i\rangle<r_i$ for all $i\in[m]$ and $x\in R$, so $\langle x,\alpha\rangle<r_{\alpha}(\I)$ for all $x\in R$. So if $\langle x_0,\alpha\rangle>t$ for some $x_0\in R$, then $r_{\alpha}(\I)>\langle x_0,\alpha\rangle>t$, so $r_{\alpha}(\underline{\I})=r_{\alpha}(\I)>t$.
 \end{proof}

\end{lemma}

\begin{lemma}\label{ind}
 If $\alpha$ is a rank $r$ indecomposable element of $\I$, then $\alpha$ is a rank $r$ indecomposable element of $\underline{\I}$.
 \begin{proof}
  Let $\alpha$ be a rank $r$ indecomposable element of $\I$. Then $\alpha\in I_r=\underline{I}_r$, and $r_{\alpha}(\underline{\I})=r_{\alpha}(\I)=r$ by Lemma \ref{r=r3}.
  We have that $\alpha\notin I_i+I_j=\underline{I}_i+\underline{I}_j$ for $i+j=r$.
  If $r_{\alpha+\beta}(\underline{\I})=t\leq k+1$, then $\alpha+\beta\in \underline{I}_t$ by Lemma \ref{inI}.
  So if $t\leq k$, we have $r_{\alpha+\beta}(\I)=r_{\alpha+\beta}(\underline{\I})$ by Lemma \ref{r=r3}.
  If $t=k+1$, then $\alpha+\beta\in I_k$ or $\alpha+\beta\in\bigcup_{i+j=k+1}((I_i+I_j)\cap\Phi^+)$, since $\alpha+\beta\notin S$.
  Either way, $\alpha+\beta\in \langle I_k\rangle_{\mathbb{N}}$ so $r_{\alpha+\beta}(\I)=r_{\alpha+\beta}(\underline{\I})$ by Lemma \ref{r=r3}.
  Thus we have $r_{\alpha}(\I)+r_{\beta}(\I)=r_{\alpha+\beta}(\I)=r_{\alpha+\beta}(\underline{\I})=t$ using Lemma \ref{rind}.
  So $r_{\beta}(\I)=t-r_{\alpha}(\I)=t-r$, so $\beta\in I_{t-r}=\underline{I}_{t-r}$ by Lemma \ref{inI}.
  Thus $\alpha$ is a rank $r$ indecomposable element of $\underline{\I}$.
 \end{proof}

\end{lemma}

\begin{lemma}\label{ceil}
 If $\alpha\in\Phi^+$ and $H_{\alpha}^r$ is a ceiling of a dominant region $R$ of the $k$-Catalan arrangement, then $\alpha$ is a rank $r$ indecomposable element of $\I=\theta(R)$.
 \begin{proof}
  Since the origin and $R$ are on the same side of $H_{\alpha}^r$, we have that $\langle x,\alpha\rangle<r$ for all $x\in R$, so $\alpha\in I_r$ and thus $r_{\alpha}(\I)\leq r$. But if $r_{\alpha}(\I)=i<r$, then $\alpha\in I_i$ by Lemma \ref{inI}, so $\langle x,\alpha\rangle<i\leq r-1$ for all $x\in R$. So $H_{\alpha}^r$ is not a wall of $R$, a contradiction. Thus $r_{\alpha}(\I)=r$.\\
  \\
  If $\alpha=\beta+\gamma$ for $\beta\in I_i$ and $\gamma\in I_j$ with $i+j=r$, then the fact that $\langle x,\alpha\rangle<r$ for all $x\in R$ is a consequence of $\langle x,\beta\rangle<i$ and $\langle x,\gamma\rangle<j$ for all $x\in R$, so $H_{\alpha}^r$ does not support a facet of $R$. So $\alpha\notin I_i+I_j$ for $i+j=r$.\\
  \\
  If $r_{\alpha+\beta}(\I)=t\leq k$, then $\alpha+\beta\in I_t$ by Lemma \ref{inI}, so $\langle x,\alpha+\beta\rangle<t$ for all $x$ in $R$. If also $\langle x,\beta\rangle>t-r$ for all $x\in R$, then $\langle x,\alpha\rangle<r$ for all $x\in R$ is a consequence of these, so $H_{\alpha}^r$ does not support a facet of $R$.
  So $\langle x,\beta\rangle<t-r$ for all $x\in R$, so $\beta\in I_{t-r}$.\\
  \\
  Thus $\alpha$ is a rank $r$ indecomposable element of $\I$.
 \end{proof}

\end{lemma}
\begin{proof}[Proof of Theorem \ref{ind=ceil}]
We take $B$ to be the pseudomaximal alcove of $R$, that is the maximal alcove of $\underline{R}$.
 We will show that (1) $\Rightarrow$ (2) $\Rightarrow$ (3) $\Rightarrow$ (1).\\
 \\
 The statement that (1) $\Rightarrow$ (2) is Lemma \ref{ceil}.\\
 \\
 For (2) $\Rightarrow$ (3), suppose $\alpha$ is a rank $r$ indecomposable element of $\I$. Then by Lemma \ref{ind}, $\alpha$ is also a rank $r$ indecomposable element of $\underline{\I}$.
 So by Lemma \ref{rind}, we have $r_{\alpha}(\underline{\I})=r_{\beta}(\underline{\I})+r_{\gamma}(\underline{\I})-1$ if $\alpha=\beta+\gamma$ for $\beta,\gamma\in\Phi^+$, and also $r_{\alpha}(\underline{\I})+r_{\beta}(\underline{\I})=r_{\alpha+\beta}(\underline{\I})$ if $\beta,\alpha+\beta\in\Phi^+$.
 Thus there exists an alcove $B'$ with $r(B',\beta)=r_{\beta}(\underline{\I})$ for $\beta\neq\alpha$ and $r(B',\alpha)=r_{\alpha}(\underline{\I})+1$ by Lemma \ref{aff}.
 Since $r(B,\beta)=r_{\beta}(\underline{\I})$ for all $\beta\in\Phi^+$, this means that $B'$ and $B$ are on the same side of each hyperplane of the affine Coxeter arrangement, except for $H_{\alpha}^{r_{\alpha}(\I)}=H_{\alpha}^r$.
 Thus $H_{\alpha}^r$ is a wall of $B$. Since $H_{\alpha}^r$ does not separate $B$ from the origin, it is a ceiling of $B$. \\
 \\
 For (3) $\Rightarrow$ (1), suppose $H_{\alpha}^r$ is a ceiling of $B$. Let $B'$ be the alcove which is the reflection of $B$ in the hyperplane $H_{\alpha}^r$.
 Then $\langle x,\alpha\rangle>r$ for all $x\in B'$, so by Lemma \ref{max} the alcove $B'$ is not contained in $R$.
 Thus $H_{\alpha}^r$ is a wall of $R$. It does not separate $R$ from the origin, so it is a ceiling of $R$.
 This completes the proof.
\end{proof}


\section{Proof of Theorem \ref{bij}}
We are now in a position to prove Theorem \ref{bij}.
\begin{proof}[Proof of Theorem \ref{bij}]
    Let us at first suppose that $\Phi$ is an irreducible crystallographic root system of rank $n$.
    For $m=0$, the statement is immediate.
    Suppose that $0<m\leq n$.\\
    \\
    To define the bijection $\Theta$, let $R\in U(M)$ and let $A$ be the minimal alcove of $R$. 
     The reflections $s^{i_1}_{\alpha_1},\ldots,s^{i_m}_{\alpha_m}$ in the hyperplanes $H^{i_1}_{\alpha_1},\ldots,H^{i_m}_{\alpha_m}$ are reflections in facets of the alcove $A=w(A_{\circ})$, so the set $S'=\{s^{i_1}_{\alpha_1},\ldots,s^{i_m}_{\alpha_m}\}$ equals $wJw^{-1}$ for some $J\subset S_a$ and $w\in W_a$.
     Thus the reflection group $W'$ generated by $S'$ is a proper parabolic subgroup of $W_a$. In particular, it is finite. 
     With respect to the finite reflection group $W'$, the alcove $A$ is contained in the dominant Weyl chamber, that is the set $$C=\{x\in V\mid\langle x,\alpha_j\rangle>i_j\text{ for all }j\in[m]\}\text{.}$$ So if $w_0'$ is the longest element of $W'$ with respect to the generating set $S'$, the alcove $A'=w_0'(A)$ is contained in the Weyl chamber $$w_0'(C)=\{x\in V\mid\langle x,\alpha_j\rangle<i_j\text{ for all }j\in[m]\}$$ of $W'$, so it is on the other side of all the hyperplanes $H^{i_1}_{\alpha_1},\ldots,H^{i_m}_{\alpha_m}$.
     $A'$ is an alcove, so it is contained in some region $R'$. Set $\Theta(R)=R'$.
\begin{figure}[h]
\begin{center}
 \includegraphics{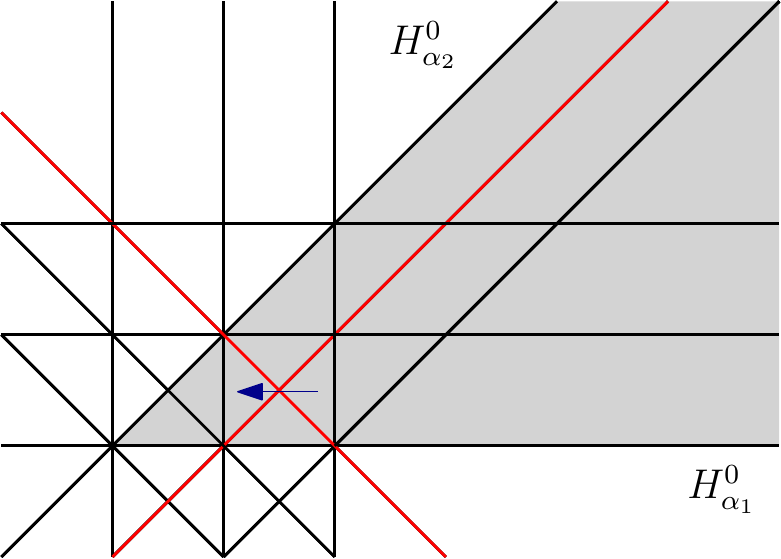}
\end{center}
\caption{The bijection $\Theta$ for the 2-Catalan arrangement of the root system of type $B_2$ with $M=\{H_{\alpha_2}^1,H_{2\alpha_1+\alpha_2}^2\}$.}
\end{figure}
    \begin{claim}\label{inlm}
     The region $R'$ is dominant and all hyperplanes in $M$ are ceilings of $R'$, that is $R'\in L(M)$, so $\Theta$ is well-defined.
     \begin{proof}
     The origin is contained in the Weyl chamber $w_0'(C)$ of $W'$.
     Thus no reflection in $W'$ fixes the origin. We can write $A'=w_0'(A)$ as $t_r\cdots t_1(A)$ where $t_i\in W'$ is a reflection in a facet of $t_{i-1}\cdots t_1(A)$ for all $i\in[r]$. In fact, if $w_0'=s_1'\cdots s_r'$ with $s_i'\in S'$ for all $i\in[r]$ is a reduced expression for $w_0'$ in $W'$, we can take $t_i=s_1'\cdots s_{i-1}'s_i's_{i-1}'\cdots s_1'$.
     So $t_i\cdots t_1(A)$ and $t_{i-1}\cdots t_1(A)$ are on the same side of every hyperplane in the affine Coxeter arrangement of $\Phi$ except for the reflecting hyperplane of $t_i$. Since $t_i$ does not fix the origin, if $t_{i-1}\cdots t_1(A)$ is dominant, then so is $t_{i}\cdots t_1(A)$. Thus by induction on $i$, the alcove $A'$ is dominant, so $R'$ is dominant.\\
     \\     
      Consider the Coxeter arrangement of $W'$, which is the hyperplane arrangement given by the reflecting hyperplanes of all the reflections in $W'$.
      The action of $W'$ on $V$ restricts to an action on the set of these hyperplanes. Since $H^{i_1}_{\alpha_1},\ldots,H^{i_m}_{\alpha_m}$ support facets of $A$, $w_0'(H^{i_1}_{\alpha_1}),\ldots,w_0'(H^{i_m}_{\alpha_m})$ support facets of $A'=w_0'(A)$.
      Now the set $\{w_0'(H^{i_1}_{\alpha_1}),\ldots,w_0'(H^{i_m}_{\alpha_m})\}$ is the set of walls of $w_0'(C)$ in the Coxeter arrangement of $W'$, so it equals the set $M=\{H^{i_1}_{\alpha_1},\ldots,H^{i_m}_{\alpha_m}\}$.
      Since all hyperplanes in $M$ are floors of $A$, and $A'$ is on the other side of each of them, they are all ceilings of $A'$.
      Thus they are ceilings of $R'$.
     \end{proof}
    \end{claim}
    We show that $\Theta$ is a bijection by exhibiting its inverse $\Psi$, a map from $L(M)$ to $U(M)$.
    Suppose $R'\in L(M)$. Let $B$ be the alcove in $R'$ given by Theorem \ref{ind=ceil}.
    Let $R''$ be the region that contains $B'=w_0'(B)$.
    Similarly to the proof of Claim \ref{inlm}, we have that $R''\in U(M)$.
    So let $\Psi(R')=R''$.
    \begin{claim}
     The maps $\Theta$ and $\Psi$ are inverse to each other, so $\Theta$ is a bijection.
     \begin{proof}
      Suppose $R\in U(M)$, $R'=\Theta(R)$ and $R''=\Psi(R')$. Use the same notation as above for the alcoves $A,A',B$ and $B'$. Suppose for contradiction that $R''\neq R$. Then there is a hyperplane $H=H^r_{\alpha}$ of the $k$-Catalan arrangement that separates $R$ and $R''$. So $H$ separates $A$ and $B'$. Now $A$ and $B'$ are in the dominant Weyl chamber of $W'$, so they are on the same side of each reflecting hyperplane of $W'$. Thus $H$ is not a reflecting hyperplane of $W'$.
      Now we may write $A'$ as $t_r\cdots t_1(A)$, where $t_i\in W'$ is a reflection in a facet of $t_{i-1}\cdots t_1(A)$ for all $i\in[r]$. So $t_i\cdots t_1(A)$ and $t_{i-1}\cdots t_1(A)$ are on the same side of every hyperplane in the affine Coxeter arrangement, except for the reflecting hyperplane of $t_i$, which cannot be $H$. Thus by induction on $i$, the alcove $A'$ is on the same side of $H$ as $A$. Similarly $B$ is on the same side of $H$ as $B'$. So $A'$ and $B$ are on different sides of $H$, a contradiction, as they are contained in the same region, namely $R'$.
      Thus $\Psi(\Theta(R))=R''=R$, so $\Psi\circ\Theta=id$. Similarly $\Theta\circ\Psi=id$, so $\Theta$ and $\Psi$ are inverse to each other, so $\Theta$ is a bijection.
     \end{proof}
    \end{claim}

    For any dominant alcove, at least one of its $n+1$ facets must either be a floor or contain the origin, and at least one must be a ceiling. So it has at most $n$ ceilings and at most $n$ floors. So any dominant region $R$ of the $k$-Catalan arrangement has at most $n$ ceilings and at most $n$ floors.
    Thus if $m>n$, both $U(M)$ and $L(M)$ are empty.
    This completes the proof in the case where $\Phi$ is irreducible.\\
    \\
    Now suppose $\Phi$ is reducible, say $\Phi=\Phi_1\amalg\Phi_2$ with $\Phi_1\perp\Phi_2$. So $V=V_1\oplus V_2$ with $V_1=\langle\Phi_1\rangle$ and $V_2=\langle\Phi_2\rangle$, and $V_1\perp V_2$.
    Then the regions of the $k$-Catalan arrangement of $\Phi$ are precisely the sets of the form $R_1\oplus R_2$ where $R_i$ is a region of the $k$-Catalan arrangement of $\Phi_i$ for $i=1,2$. The region $R_1\oplus R_2$ is dominant if and only if $R_1$ and $R_2$ are both dominant.
    A hyperplane $H_{\alpha}^r$ is a floor of $R_1\oplus R_2$ if and only if $H_{\alpha}^r$ is a floor of $R_i$ for some $i=1,2$.
    The same holds for ceilings. Say $M=M_1\amalg M_2$ with $H_{\alpha_j}^{i_j}\in M_i$ if $\alpha_j\in\Phi_i$ for $j\in[m]$ and $i=1,2$. Assume the theorem holds for $\Phi_1$ and $\Phi_2$, giving us bijections $\Theta_1$ and $\Theta_2$ for $\Phi_1$ together with $M_1$ and $\Phi_2$ together with $M_2$ respectively. 
    Then $\Theta(R_1\oplus R_2)=\Theta_1(R_1)\oplus \Theta_2(R_2)$ gives the required bijection for $\Phi$ together with $M$. This completes the proof by induction on the number of irreducible components of $\Phi$.
 \end{proof}
\section{Corollaries}
We deduce some enumerative corollaries of Theorem \ref{bij}.
For any set $M$ of hyperplanes of the $k$-Catalan arrangement, let $U_=(M)$ be the set of dominant regions $R$ of the $k$-Catalan arrangement such that the floors of $R$ are exactly the hyperplanes in $M$, and let $L_=(M)$ be the set of dominant regions $R'$ of the $k$-Catalan arrangement such that the ceilings of $R'$ are exactly the hyperplanes in $M$. 
%
\begin{corollary}\label{inex}
  For any set $M=\{H_{\alpha_1}^{i_1},H_{\alpha_2}^{i_2},\ldots,H_{\alpha_m}^{i_m}\}$ of $m$ hyperplanes with $i_j\in [k]$ and $\alpha_j\in\Phi^+$ for all $j\in[m]$, we have that $|U_=(M)|=|L_=(M)|$.
  \begin{proof}
   This follows from Theorem \ref{bij} by an application of the Principle of Inclusion and Exclusion.
  \end{proof}
\end{corollary}
\begin{corollary}\label{sum}
  For any tuple $(a_1,a_2,\ldots,a_k)$ of nonnegative integers, the number of dominant regions $R$ that have exactly $a_j$ floors of height $j$ for all $j\in[k]$ is the same as the number of dominant regions $R'$ that have exactly $a_j$ ceilings of height $j$ for all $j\in[k]$.
  \begin{proof}
   Sum Corollary \ref{inex} over all sets $M$ containing exactly $a_j$ hyperplanes of height $j$ for all $j\in[k]$.
  \end{proof}
\end{corollary}
  \begin{proof}[Proof of Corollary \ref{arm}]
   Set $a_r=l$ and sum Corollary \ref{sum} over all choices of $a_j$ for all $j\neq r$.
  \end{proof}
%

\section{The Panyushev complement}

In the special case where $k=1$, a geometric chain of ideals $\I$ is simply the single ideal $I_1$, similarly a geometric chain of order filters $\J$ is just the single order filter $J_1$.
The indecomposable elements of an ideal $I$ are then just its maximal elements \cite[Lemma 3.9]{athanasiadis06cluster}.
The indecomposable elements of an order filter $J$ are just its minimal elements \cite[Lemma 3.9]{athanasiadis05refinement} \cite[Lemma 1]{thiel13hf}.\\
\\
There is a natural bijection between ideals and antichains of any poset that sends an ideal to the set of its maximal elements.
Similarly, there is a natural bijection between order filters and antichains that sends an order filter to the set of its minimal elements.\\
\\
So for an ideal $I$ in the root poset of $\Phi$, we define the Panyushev complement $\mathbf{Pan}(I)$ as the ideal generated by the minimal elements of the order filter $J=\Phi^+\backslash I$. From the above considerations, this is a bijection from the set of order ideals of the root poset of $\Phi$ to itself.\\
\\
For a region $R$ of the Catalan arrangement, let 
$$CL(R)=\{\alpha\in\Phi^+\mid H_{\alpha}^1\text{ is a ceiling of $R$}\}\text{ and}$$
$$FL(R)=\{\alpha\in\Phi^+\mid H_{\alpha}^1\text{ is a floor of $R$}\}\text{.}$$
Since a region $R$ in the Catalan arrangement corresponds to a unique ideal $I=\theta(R)$, which corresponds uniquely to the set of its maximal elements, which equals $CL(R)$ by Theorem \ref{ind=ceil}, the map $CL:R\mapsto CL(R)$ gives a bijection from the set of dominant regions in the Catalan arrangement to the set of antichains in the root poset.
That the same holds for the map $FL:R\mapsto FL(R)$ follows from an analogous argument that can already be deduced from \cite[Theorem 3.11]{athanasiadis05refinement}.
\begin{figure}[h]
\begin{center}
 \includegraphics{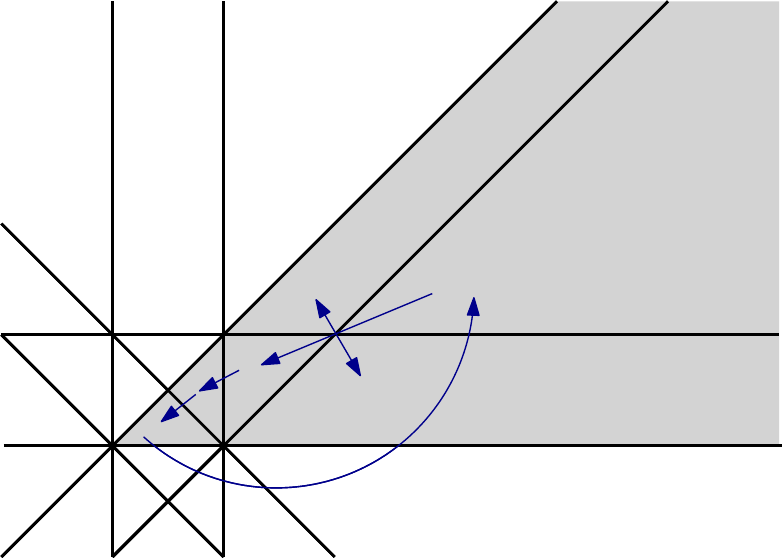}\\
\end{center}

\caption{The action of $\theta^{-1}\circ\mathbf{Pan}\circ\theta=CL^{-1}\circ FL$ on the dominant regions of the Catalan arrangement of the root system of type $B_2$.}
\end{figure}
\begin{theorem}\label{pan}
 For an ideal $I$ in the root poset of $\Phi$, the region $\theta^{-1}(\mathbf{Pan}(I))$ is the unique region of the Catalan arrangement of $\Phi$ whose ceilings are exactly the floors of the region $\theta^{-1}(I)$.
 \begin{proof}
  The set $CL(\theta^{-1}(\mathbf{Pan}(I)))$ is the set of maximal elements of $\mathbf{Pan}(I)$, which equals the set of minimal elements of $J=\Phi^+\backslash I$, which equals $FL(\theta^{-1}(I))$.
  Since $CL$ is a bijection, $\theta^{-1}(\mathbf{Pan}(I))$ is the only region $R'$ with $CL(R')=FL(\theta^{-1}(I))$.
 \end{proof}

\end{theorem}

We could rephrase Theorem \ref{pan} as $\mathbf{Pan}=\theta\circ\ CL^{-1}\circ FL\circ\theta^{-1}$. The fact that the Panyushev complement has a natural interpretation in terms of the dominant regions of the Catalan arrangement may serve to explain why it seems to be of particular interest for root posets.

\section{Acknowledgements}
I thank Myrto Kallipoliti, Henri M{\"u}hle, Vivien Ripoll and Eleni Tzanaki for helpful comments and suggestions.
I also wish to express my gratitude to the anonymous referee for a very careful reading of the manuscript.

\bibliographystyle{alpha}
\bibliography{literature}

\begin{thebibliography}{FKT13}

\bibitem[Arm09]{armstrong09generalized}
Drew Armstrong.
\newblock {Generalized Noncrossing Partitions and Combinatorics of Coxeter
  Groups}.
\newblock {\em Memoirs of the American Mathematical Society}, 202, 2009.

\bibitem[AT06]{athanasiadis06cluster}
Christos~A. Athanasiadis and Eleni Tzanaki.
\newblock {On the enumeration of positive cells in generalized cluster
  complexes and Catalan hyperplane arrangements}.
\newblock {\em Journal of Algebraic Combinatorics}, 23:355--375, 2006.

\bibitem[Ath04]{athanasiadis04generalized}
Christos~A. Athanasiadis.
\newblock {Generalized Catalan Numbers, Weyl Groups and Arrangements of
  Hyperplanes}.
\newblock {\em Bulletin of the London Mathematical Society}, 36:294--302, 2004.

\bibitem[Ath05]{athanasiadis05refinement}
Christos~A. Athanasiadis.
\newblock {On a Refinement of the Generalized Catalan Numbers for Weyl Groups}.
\newblock {\em Transactions of the American Mathematical Society},
  357:179--196, 2005.

\bibitem[FKT13]{fishel13facets}
Susanna Fishel, Myrto Kallipoliti, and Eleni Tzanaki.
\newblock {Facets of the generalized cluster complex and regions in the
  extended Catalan arrangement of type $A_n$}.
\newblock {\em The Electronic Journal of Combinatorics}, 20, 2013.

\bibitem[FR05]{fomin05generalized}
Sergey Fomin and Nathan Reading.
\newblock {Generalized Cluster Complexes and Coxeter Combinatorics}.
\newblock {\em International Mathematics Research Notices}, 44:2709--2757,
  2005.

\bibitem[FTV13]{fishel13fixed}
Susanna Fishel, Eleni Tzanaki, and Monica Vazirani.
\newblock {Counting Shi regions with a fixed separating wall}.
\newblock {\em Annals of Combinatorics}, 17:671--693, 2013.

\bibitem[Hum90]{humphreys90reflection}
James~E. Humphreys.
\newblock {\em {Reflection Groups and Coxeter Groups}}.
\newblock Cambridge University Press, Cambridge, 1990.

\bibitem[Shi87]{shi87alcoves}
Jian{-}yi Shi.
\newblock {Alcoves corresponding to an affine Weyl group}.
\newblock {\em Journal of the London Mathematical Society}, 35:42--55, 1987.

\bibitem[Thi14]{thiel13hf}
Marko Thiel.
\newblock {On the $H$-triangle of generalised nonnesting partitions}.
\newblock {\em European Journal of Combinatorics}, 39:244--255, 2014.

\end{thebibliography}

%
%

\end{document}